\documentclass[12pt,notitlepage,a4paper]{article}

\usepackage{amsmath,amsfonts,amsthm,bbm}
\usepackage[arrow,matrix]{xy}
\usepackage{enumerate}

\theoremstyle{plain}
\numberwithin{equation}{section}
\newtheorem{thm}{Theorem}[section]
\newtheorem{cor}[thm]{Corollary}
\newtheorem{lem}[thm]{Lemma}
\newtheorem{prop}[thm]{Proposition}

\theoremstyle{definition}
\newtheorem{df}[thm]{Definition}
\newtheorem{ex}[thm]{Example}
\newtheorem{rmk}[thm]{Remark}

\renewcommand{\proofname}{\bf Proof}

\newcommand{\mb}{\mathbb}
\newcommand{\mf}{\mathfrak}
\newcommand{\ml}{\mathcal}

\newcommand{\Q}{\mb{Q}}
\newcommand{\Z}{\mb{Z}}

\newcommand{\ch}{{\rm CH}}
\newcommand{\h}{H^S}
\newcommand{\DM}{\boldsymbol{\rm DM}^{\rm eff}_{\rm -, Nis}}
\newcommand{\HomD}{{\rm Hom}_{\DM}}
\newcommand{\sym}{{\rm Sym}}
\renewcommand{\ker}{{\rm Ker}}
\newcommand{\im}{{\rm Im}}
\newcommand{\tensor}{\otimes}
\newcommand{\Cor}{\mathrm{Cor}}
\newcommand{\M}{{\rm M}}

\newcommand{\PST}{\boldsymbol{\rm PST}}
\newcommand{\ShT}{\boldsymbol{\rm ShT}_{\rm Nis}}

\newcommand{\qfh}{{\rm \bf Sh}_{\rm qfh}(k,\Q)}
\newcommand{\nis}{\underline{\rm \bf ShT}_{\rm Nis}(k,\Q)}

\begin{document}
\author{Rin Sugiyama}
\title{Motivic homology of semiabelian varieties}

\maketitle

\begin{abstract}
We generalize some classical results on Chow group of an abelian variety to semiabelian varieties and to motivic (co)homology, using a result of Ancona--Enright-Ward--Huber \cite{AEH} on a decomposition of the motive of a semiabelian variety in the Voevodsky's category.
\end{abstract}

\noindent{Mathematics Subject Classification(2010)}: 19E15,14C25

\noindent{Keywords}: semiabelian variety, motivic homology, algebraic cycles

\section{Introduction}
Let $k$ be a field.
Let $G$ be a semiabelian variety over $k$.
By the definition of semiabelian varieties, $G$ is an extension of an abelian variety $A$ of dimension $g$ by a torus $T$ of rank $r$, i.e. there is the following exact sequence of smooth group schemes over $k${\rm :}
\begin{align}
1\to T \to G \to A \to 1.
\end{align}
We say that $G$ is of rank $r$ if the rank of $T$ is $r$.
A semiabelian variety of rank zero means an abelian variety.

In this paper, we study of structure of motivic homology $H_p(G,\Q(q))$ (cf. Notation) and Bloch's higher Chow group $\ch^q(G,s;\Q)$ (\cite{B2}) (motivic cohomology, see Theorem \ref{voe}).
For an abelian variety $A$, Beauville \cite{Be1} proved the following theorem (cf. \cite{DM,K}):

\begin{thm}[Beauville \cite{Be1}]\label{tbe}
Let $A$ be an abelian variety of dimension $g$ over a field $k$.
Let $q$ be an integer with $0\leq q\leq g$.
For an integer $i$, let $\ch^q(G;\Q)^{(i)}$ denote the subspace of the Chow group $\ch^q(A;\Q):$
\[
\ch^q(A;\Q)^{(i)}:=\{\alpha \in \ch^q(A;\Q) \ |\ n_A^*\alpha=n^i\alpha \ \text{for all}\ n\}.
\]
Here $n_A^*$ is the pull-back along the map $A \stackrel{\times n}{\longrightarrow}A$ of multiplication by an integer $n$.
Then $\ch^q(A;\Q)^{(i)}$ is zero for $i \not \in [q,g+q]$ and the Chow group $\ch^q(A;\Q)$ is decomposed as follows$:$ for $0\leq q \leq g$,
\[
\ch ^q(A;\Q)= \bigoplus_{i=q}^{g+q}\ch^q(A;\Q)^{(i)}.
\]
\end{thm}
Note that the index is different from one of Beauville \cite{Be1} (his notation $\ch_s^p(A)$ corresponds to our $\ch^p(A;\Q)^{(2p-s)}$).

Bloch \cite{B} studied an iterated Pontryagin product $I_A^{*r}$ of the kernel $I_A$ of the degree map $\ch_0(A) \to \Z$ for an abelian variety $A$ over an algebraically closed field, and he proved the following theorem:

%
\begin{thm}[Bloch \text{\cite{B}}]
Let $A$ be an abelian variety of dimension $g$ over an algebraically closed field.
Then $I_A^{*i}=0$ for $i>g$.
\end{thm}

\medskip
In this paper, we generalize the above theorems to semiabelian varieties.

Let $G$ be a semiabelian variety over $k$.
For an integer $n$, let $n_G$ denote the map $G \stackrel{\times n}{\longrightarrow} G$ of multiplication by $n$.
 For an integer $i$, we define the subgroup $\ch^q(G,s;\Q)^{(i)}$ of the higher Chow group $\ch^q(G,s;\Q)$ as follows (cf. Notation):
\[
\ch^q(G,s;\Q)^{(i)}:=\{\alpha \in \ch^q(G,s;\Q) \ |\ n_G^*\alpha=n^i\alpha \ \text{for all}\ n\}.
\]
Here $n_G^*$ is the pull-back along $n_G$.
Similarly, for an integer $i$, we define the subgroup $H_p(G,\Q(q))_{(i)}$ of the motivic homology $H_p(G,\Q(q))$ as follows:
\[
H_p(G,\Q(q))_{(i)}:=\{\alpha \in H_p(G,\Q(q)) \ |\ {n_G}_*\alpha=n^i\alpha \ \text{for all}\ n\}.
\]
Here ${n_G}_*$ denotes the push-forward map along $n_G$. 
If $G$ is an abelian variety $A$, then $H_{2q}(A,\Q(q))$ is isomorphic to the Chow group $\ch_q(A;\Q)$ of algebraic cycles of dimension $q$ on $A$ modulo rational equivalence with $\Q$-coefficients.


\begin{thm}[Corollary \ref{decm}]\label{mth}
Let $G$ be a semiabelian variety over a perfect field $k$, which is an extension of an abelian variety of dimension $g$ by a torus of rank $r$.
\begin{enumerate}[{\rm (1)}]
\item
We have the following natural decomposition
\[
\ch^q(G,s;\Q)=\bigoplus_{i=j_0}^{j_1}\ch^q(G,s;\Q)^{(i)}
\]
where $j_0=\max\{ 0,q-s\}$ and $j_1=\min\{2g+r, g+q\}$.

\item
We have the following natural decomposition 
\[
H_p(G,\Q(q))=\bigoplus_{i=i_0}^{i_1}H_p(G,\Q(q))_{(i)}
\]
where $i_0=\max\{0,q\}$ and $i_1=\min\{2g+r,g+r+p-q\}$.
\end{enumerate}
\end{thm}


\begin{rmk}
(a)
For an abelian variety, Theorem \ref{mth}\,(1) and (2) are the same.
Theorem \ref{mth}\,(1) for $s=0$ is Beauville's result (Theorem \ref{tbe}).

(b)
Let notation be as in Theorem \ref{mth}.
Let $q\in [0,g]$ be an integer.
Beauville \cite{Be1} has conjectured that for $i>2q$,
\[
\ch^q(A;\Q)^{(i)}=0.
\]
This conjecture implies that
$
H_{2q-r}(G,\Q(q))_{(i)}=0
$
for $i<2q-r$
and that
$
\ch^q(G;\Q)^{(i)}=0
$
for $i>2q$.
\end{rmk}

\medskip
We denote the $0$-th Suslin homology of $G$ with $\Z$-coefficient by $\h_0(G,\Z)$.
Using Theorem \ref{mth}\,(2) in case that $p=q=0$ and results on Rojtman's theorem (see Theorem \ref{rt}) by Spiess--Szamuely \cite{SS} and Geisser \cite{Ge1,Ge2}, we obtain Bloch's result for semiabelian varieties and 0-th Suslin homology:
\begin{cor}[Corollary \ref{shom}]\label{mthm}
Let $G$ be a semiabelian variety over $k$, which is an extension of an abelian variety of dimension $g$ by a torus of rank $r$.
Let $I_G$ be the kernel of the degree map ${\rm deg}: \h_0(G,\Z)\to \Z$.
Then an iterated Pontryagin product $I_{G}^{*i}$ is torsion for $i>g+r$.

In particular, if $k$ is an algebraically closed field of characteristic $p\geq 0$, then for $i>g+r$,
\[
I_G^{*i}=0
\] 
up to $p$-torsion part.
Under the resolution of singularities, it holds without the assumption on the characteristic.
\end{cor}

\begin{rmk}
The subgroups $H_0(G,\Q)_{(i)}$ have a description in terms of $K$-groups attached to the semiabelian variety $G$ (see Proposition \ref{repK}).
\end{rmk}

\begin{rmk}\label{repft}
One can remove the assumption that a base field $k$ is perfect in the above theorems, by a result of Cisinski-D\'eglise on a comparison of Beilinson motives $\boldsymbol{\mathrm{DM}}_{\mf{b}}$ (see \cite[\S 14]{CD}) with Voevodsky motives:
If $k$ is not perfect, then we take the inseparable closure $k^\mathrm{insep}$ of $k$.
By \cite[Proposition 2.1.9, Theorem 14.3.3]{CD}, the pull-back functor
\[
\boldsymbol{\mathrm{DM}}_{\mf{b}}(k^\mathrm{insep}) \to \boldsymbol{\mathrm{DM}}_{\mf{b}}(k)
\]
is an equivalence.
By \cite[Theorem 16.1.4]{CD}, the above categories of Beilinson motives are equivalent to $\boldsymbol{\mathrm{DM}}(k^\mathrm{insep},\Q)$ and $\boldsymbol{\mathrm{DM}}(k,\Q)$ respectively.
\end{rmk}

\medskip
Theorem \ref{mth} is deduced form a result on vanishing of motivic (co)homology (Theorem \ref{vmh}) and a description of the motive $\M(G)$ in $\DM(k,\Q)$ in terms of symmetric products:
\[
\varphi_G:\M(G) \stackrel{\simeq}{\longrightarrow} \sym\bigl(\M_1(G)\bigr)
\]
which is due to Ancona--Enright-Ward--Huber \cite{AEH} (see Theorem \ref{mdec}, Corollary \ref{dec}).
Here $\M_1(G)$ denotes the complex consisting of the homotopy invariant Nisnevich sheaf $\widetilde{G}:={\rm Hom}_{{\rm Sch}/k}(-,G)\tensor \Q$ with transfers concentrated in degree zero.

There are two key ingredients in proof of vanishing of motivic (co)homology.
One is an exact triangle which relates $G$ to a semiabelian variety of rank ${\rm rk (G)}-1$ (Lemma \ref{t2}, Lemma \ref{keyt}).
The exact triangle allows us to use induction on rank of $G$.
Another is to interpret Weil-Barsotti formula for an abelian variety in terms of motives in $\DM$.
Using this interpretation, one can generalize Beauville's result to higher Chow groups, i.e.\ Theorem \ref{mth} for an abelian variety.
\bigskip

This paper is organized as follows:
In Section 2, we recall the result (Theorem \ref{mdec}) on the description of the motive $\M(G)$ of a semiabelian variety $G$ in $\DM(k,\Q)$.
We give a consequence of the description which is a decomposition of motivic (co)homology of semiabelian varieties. 

In Section 3, we state and prove our main result on vanishing of motivic (co)homology (Theorem \ref{vmh}) and a consequence of main result (Theorem \ref{mth}).
We first prove main result in case of a torus and a abelian variety.
Lastly, we prove by induction for general case, using key triangles (Lemma \ref{t2}, Lemma \ref{keyt}) in $\DM$.

In Section 4, we consider $0$-th Suslin homology and prove Corollary \ref{mthm} (Corollary \ref{shom}).
We also give a description of rational $0$-th Suslin homology of a semiabelian variety $G$ in terms of $K$-groups attached to $G$ (Proposition \ref{repK}, Remark \ref{shk}).

In Appendix A, we prove a lemma (Lemma \ref{skey}) about an exact sequence which is related to symmetric products in an abelian tensor category.
This lemma is used to get a key triangle (Lemma \ref{keyt}) in $\DM$.

In Appendix B, we make some remarks about the tensor product in the category $\ShT(k,\Q)$ of Nisnevich sheaves of $\Q$-modules with transfers.
Main part is to show the exactness of the tensor product.
The exactness is used to apply Lemma \ref{skey} to $\ShT(k,\Q)$.

\subsection*{Notation}
\begin{enumerate}
\item
Let $k$ be a perfect field.
For a commutative ring $R$, let $\DM(k,R)$ denote the Voevodsky's tensor triangulated category of effective motives over a perfect field $k$ with $R$-coefficients.
Let ${\rm \bf DM}_{-, \rm Nis}(k,R)$ denote the Voevodsky's tensor triangulated category of motives over $k$ with $R$-coefficients (cf. \cite{V}).
In case $R=\Q$, we simply write  $\DM$ (resp. ${\rm \bf DM}_{-, \rm Nis}$) for  $\DM(k,\Q)$ (resp. ${\rm \bf DM}_{-, \rm Nis}(k,\Q)$).

\item Let $X$ be a smooth scheme of finite type over $k$.
Let $R$ be a commutative ring.
We denote the motive of $X$ in $\DM(k,R)$ by $\M(X)$.
For any integers $p$ and $q$, we denote the motivic (co)homology of $X$ with $R$-coefficients by
\begin{align*}
H_p(X,R(q))&:={\rm Hom}_{{\rm \bf DM}_{-,\rm Nis}(k,R)}(R(q)[p],\M(X));\medskip\\
H^p(X,R(q))&:={\rm Hom}_{{\rm \bf DM}_{-,\rm Nis}(k,R)}(\M(X),R(q)[p]).
\end{align*}
\begin{thm}[\text{Voevodsky \cite{V2}}]\label{voe}
Motivic cohomology of smooth scheme $X$ over a field $k$ is isomorphic to higher Chow group for any coefficients $R:$ for any $p,q \in \Z$,
\begin{align}
H^p(X,R(q))\simeq \ch^q(X,2q-p;R).
\end{align}
\end{thm}
\item
For an integer $i\geq0$, we define the subgroup $H_p(G,\Q(q))^{(i)}$ (resp. $H^p(G,\Q(q))^{(i)}$) of $H_p(G,\Q(q))$ (resp. $H^p(G,\Q(q)))$) as follows:
\begin{align*}
H_p(G,\Q(q))_{(i)}&:=\{\alpha \in H_p(G,\Q(q)) \ |\ n_{G*}\alpha=n^i\alpha \ \text{for all}\ n\},\\
H^p(G,\Q(q))^{(i)}&:=\{\alpha \in H_p(G,\Q(q)) \ |\ n_G^*\alpha=n^i\alpha \ \text{for all}\ n\}.
\end{align*}
Here ${n_G}$ denotes the map $G \stackrel{\times n}{\longrightarrow} G$ of multiplication by an integer $n$.

\item 
For a positive integer $n$, let $\Sigma_n$ denote the group of permutations on $n$ letters.
Let $\ml{C}$ be a idempotent complete $\Q$-linear symmetric monoidal  category, and let $M$ be an object of $\ml{C}$.
Then we have a representation $\Sigma_n \to {\rm End}(M^{\tensor n})$, and let $\sigma_M $ denote the endomorphism of $M$ corresponding to $\sigma \in \Sigma_n$.
We define $s_M^n$ as follows:
\begin{align*}
s_M^n:=
\begin{cases}
\qquad {\rm id}_{\boldsymbol{\mathbbm{1}}} &\text{if $n=0$},\\
\displaystyle \frac{1}{n!}\sum_{\sigma \in \Sigma_n} \sigma_M &\text{if $n>0$}.
\end{cases}
\end{align*}
Here $\boldsymbol{\mathbbm{1}}$ denotes the unit object of $\ml{C}$.
Then one can easily see that $s_M^n$ is idempotent.
We define the $n$-th symmetric product $\sym^nM$ of $M$ as follows:
\begin{align*}
\sym^nM:=\im (s_M^n).
\end{align*}
This image exists since $\ml{C}$ is idempotent complete.
We write the canonical projection $M^{\tensor n}\to \sym^n M$ by $s_M^n$, and the canonical embedding $\sym^n M \to M^{\tensor n}$ by $\iota_M^n$.
\end{enumerate}


\section{The motive of a semiabelian variety}


\subsection{The isomorphism $\boldsymbol{\varphi_G}$}
We here recall a recent result of a decomposition of the motive of semiabelian varieties in $\DM(k,\Q)$, which is  due to Enright-Ward \cite{EW}/Ancona--Enright-Ward--Huber \cite{AEH}.

\begin{df}
Let $G$ be a semiabelian variety over $k$.
Then the \'etale sheaf $\widetilde{G}:={\rm Hom}_{{\rm Sch}/k}(-,G)\tensor \Q$ has a canonical structure of \'etale sheaf with transfers, and furthermore $\widetilde{G}$ is homotopy invariant (\cite[proof of Lemma 3.2]{SS} \cite[Lemma 1.3.2]{BK}).

Following \cite{AEH,EW}, we define $\M_1(G)$ to be the complex in $\DM$ consisting of $\widetilde{G}$ concentrated in degree zero.
\end{df}

\textbf{Construction of $\boldsymbol{\varphi_G}$:}
According to Spiess--Szamuely \cite{SS}, we have  morphism $\Z_{\rm tr}(G) \to \widetilde{G}$ of \'etale sheaves with transfers, and therefore there is a morphism in $\DM$:
\[
a_G: \M(G) \to \M_1(G).
\]
Using this map, we define $\varphi_G^i : \M(G) \to \sym^i(\M_1(G))$ to be the composite morphism in $\DM$:
\[
\xymatrix{
\M(G)\ar[r]^{\M(\Delta_G^i)}&\ \M(G)^{\tensor i} \ar[r]^{a_G^{\tensor i}}& \M_1(G)^{\tensor i} \ar[r]^{s_{\M_1(G)}^i}&\ \ \sym^i(\M_1(G)).
}\]

\begin{df}[Kimura \cite{Ki}]
Let $M$ be an object of $\DM$.
We say that $M$ has \textit{finite Kimura dimension} if $\sym^n(M)=0$ for some $n\geq0$.
\end{df}
In \cite{EW} and \cite{AEH}, they showed Kimura finite dimensionality of $\M_1(A)$ for an abelian variety $A$, using Kimura finite dimensionality of the Chow motive $h^1(A)$(\cite{Ki}, \cite{K}).
By Lemma \ref{keyt} and the result in case of abelian varieties, they showed Kimura finite dimensionality of $\M_1(G)$:

\begin{prop}[\text{\cite[Proposition 5.1.1]{AEH},\cite[Lemma 5.7.4]{EW}}]\label{vm}
Let $G$ be a semiabelian variety over $k$, which is an extension of an abelian variety of dimension $g$ by a torus of rank $r$.
Then $\sym^i(\M_1(G))$ is zero in $\DM$ for $i>2g+r$.
In particular, $\sym^{sg+r}(\M_1(G))\simeq \Lambda(g+r)[2g+r]$ with a tensor-invertible Artin motive $\Lambda$.
If the torus part of $G$ is split,  then $\Lambda=\Q$ $($the unit motive$)$.
\end{prop}

Let $\sym (\M_1(G))$ denote the direct sum $\displaystyle \bigoplus_{i\geq0} \sym^i(\M_1(G))$.
Here the sum is finite from Proposition \ref{vm}.
Then we define 
\begin{align}\label{iso}
\varphi_G:=\bigoplus \varphi_G^i:\M(G) \to \sym(\M_1(G)).
\end{align}

The motive $\M(G)$ has a canonical Hopf algebra structure defined by morphisms of schemes: 
\begin{itemize}
\item
the multiplication by the group law $m_G:G\times G\to G$;
\item
the comultiplication by the diagonal map $\Delta_G : G\to G\times G$;
\item
the antipodal map by the inverse on $G$;
\item
the unity by the neutral element;
\item
the counit by the structure map $G \to \mathrm{Spec}(k).$
\end{itemize}

\begin{df}
We define a Hopf algebra structure on $\sym (\M_1(G))$ as follows:
\begin{description}
\item[(multiplication)] For any $i,j\geq 0$,
\[
\frac{(i+j)!}{i!j!}s_{\M_1(G)}^{i+j}\circ (\iota_{\M_1(G)}^{i}\tensor\iota_{\M_1(G)}^{j}):\sym^i(\M_1(G))\tensor \sym^j(\M_1(G))\to\sym^{i+j}(\M_1(G));
\]
\item[(comultiplication)] For any $i,j\geq 0$,
\[
(s_{\M_1(G)}^{i}\tensor s_{\M_1(G)}^{j})\circ\iota_{\M_1(G)}^{i+j} :\sym^{i+j}(\M_1(G))\to\sym^i(\M_1(G))\tensor \sym^j(\M_1(G)).
\]
\end{description}
The antipodal map, the unity and the counit are induced by the inverse map, the unit map and the structure map of $\M_1(G)$.
\end{df}


\begin{rmk}
The definition of bialgebra structure on $\sym $ in the above definition is not familiar one.
The reasons why we use the above definition are to fit into the classical result on motivic decomposition of Chow motive of abelian varieties (cf. \cite{K}) and to make the map \eqref{iso} to be an isomorphism of Hopf algebra objects in $\DM$.
For a detail see \cite{AEH}.
\end{rmk}

\begin{thm}[Ancona--Enright-Ward--Huber \text{\cite[Theorem 7.1.1]{AEH}}]\label{mdec}
Let $G$ be a semiabelian variety over a perfect field $k$ and let $\M(G)$ be the motive of $G$ in $\DM(k,\Q)$.
Then there exists a natural isomorphism of Hopf algebras in $\DM(k,\Q)${\rm :}
\[
\varphi_G:\M(G) \stackrel{\simeq}{\longrightarrow} \sym\bigl(\M_1(G)\bigr).
\]
In particular, the following diagrams commute{\rm :}
\begin{align}\label{comp}
\xymatrix{
\M(G)\tensor\M(G) \ar[d]_{\varphi_G\tensor\varphi_G}^{\simeq}\ar[r]^{\M(m_G)}& \M(G)\ar[d]_{\simeq}^{\varphi_G}\\
\sym\bigl(\M_1(G)\bigr)\tensor\sym\bigl(\M_1(G)\bigr)\ar[r]& \sym\bigl(\M_1(G)\bigr)
;}\\
\quad \xymatrix{
 \M(G)\ar[d]_{\simeq}^{\varphi_G} \ar[r]^{\M(\Delta_G)}&\M(G)\tensor\M(G)\ar[d]_{\varphi_G\tensor\varphi_G}^{\simeq}\\
\sym\bigl(\M_1(G)\bigr)\ar[r]& \sym\bigl(\M_1(G)\bigr)\tensor\sym\bigl(\M_1(G)\bigr)
.}
\end{align}
\end{thm}

\subsection{Decomposition of motivic (co)homology of a semiabelian variety}

We give a decomposition of motivic (co)homology and a relationship between the decomposition and product structure, which  is a consequence of Theorem \ref{mdec}.
We first introduce product on motivic (co)homology.

\begin{df}
(1)
Let $a$ and $b$ be elements in $H_p(G,\Q(q))$ and $H_{p'}(G,\Q(q'))$ respectively.
Then we define \textit{the Pontryagin product} $a*b$ of $a$ and $b$ to be the image of $(a,b)$ under the morphism
\begin{align}\label{m1}
H_p(G,\Q(q))\times H_{p'}(G,\Q(q')) \to H_{p+p'}(G\times G,\Q(q+q')) \stackrel{m_{G*}}{\to} H_{p+p'}(G,\Q(q+q')).
\end{align}
Here the first map is induced by the tensor structure in $\DM$.
For any subgroups $F\subset H_p(G,\Q(q))$ and $G\subset H_{p'}(G,\Q(q'))$, we write $F*G$ for the image of $F\times G$ under morphism \eqref{m1}. 

\medskip
(2)
For $a \in H^p(G,\Q(q))$ and $b \in H^{p'}(G,\Q(q'))$, we define \textit{the cup product} $a\cup b$ to be the image of $(a,b)$ under the morphism
\begin{align}\label{m2}
H^p(G,\Q(q))\times H^{p'}(G,\Q(q')) \to H^{p+p'}(G\times G,\Q(q+q')) \stackrel{\Delta_{G}^*}{\to} H^{p+p'}(G,\Q(q+q')).
\end{align}
For any subgroups $F\subset H^p(G,\Q(q))$ and $G\subset H^{p'}(G,\Q(q'))$, we write $F\cup G$ for the image of $F\times G$ under morphism \eqref{m2}. 
\end{df}


\begin{cor}\label{dec}
Let $G$ be a semiabelian variety over $k$, which is an extension of an abelian variety of dimension $g$ by a torus of rank $r$.
Let $p,q$ be integers.
Then we have the following natural decomposition of motivic (co)homology of $G${\rm :}
\[
H_p(G,\Q(q))=\bigoplus_{i=0}^{2g+r}H_p(G,\Q(q))_{(i)}, \quad H^p(G,\Q(q))=\bigoplus_{i=0}^{2g+r}H^p(G,\Q(q))^{(i)}
\]
which satisfies
\begin{align*}
H_p(G,\Q(q))_{(i)}*H_{p'}(G,\Q(q'))_{(i')} &\subset H_{p+p'}(G,\Q(q+q'))_{(i+i')},\\
H^p(G,\Q(q))^{(i)}\cup H^{p'}(G,\Q(q'))^{(i')} &\subset H^{p+p'}(G,\Q(q+q'))^{(i+i')}.
\end{align*}
\end{cor}

\begin{proof}
The map $n_G$ on $\sym^i(\M_1(G))$ is $n^i\cdot \mathrm{id}_{\sym^i(\M_1(G))}$.
Thus we have
\begin{align*}
H_p(G,\Q(q))_{(i)}=\HomD(\Q(q)[p],\sym^i(\M_1(G))),\\
H^p(G,\Q(q))^{(i)}=\HomD(\sym^i(\M_1(G)),\Q(q)[p]).
\end{align*}
The assertion follows from this description and Theorem \ref{mdec}.
\end{proof}


\section{Main result}
In this section, we state main result on a vanishing of motivic (co)homology of a semiabelian variety and its concequenses.
%

\begin{thm}\label{vmh}
Let $G$ be a semiabelian variety over $k$, which is an extension of an abelian variety of dimension $g$ by a torus of rank $r$.
Let $p,q$ be integers.
Then
\begin{enumerate}[{\rm (1)}]
\item 
$H_p(G,\Q(q))$ vanishes in the following cases
\begin{enumerate}[{\rm (a)}]
\item $p<q;$
\item $g+r<q;$
\item $p-2q+r<0;$
\item $q=g+r$ and $p-2q+r\geq 1;$
\item $q=g+r-1$ and $p-2q+r\geq 2$.
\end{enumerate}
\item 
Let $i_0=\max\{0,q\}$ and $i_1=\min\{2g+r,g+r+p-q\}$.
Then $H_p(G,\Q(q))_{(i)}$ vanishes for $i \not \in [i_0,i_1]$.
\item 
Let $j_0=\max\{ 0,q-s\}$ and $j_1=\min\{2g+r, g+q\}$.
$\ch^q(G,s;\Q)^{(i)}$ vanishes for $i\not \in [j_0,j_1]$.
\end{enumerate}
\end{thm}


\begin{rmk}
By the definition of higher Chow groups and a computation of codimension one cycles (see Bloch \cite{B2}), we know the following result on vanishing of higher Chow groups:

\begin{thm}\label{coh}
For any smooth scheme $X$of dimension $d$ over a field $k$ and for any abelian group $R$, we have $\ch^i(X,n;R)=0$ in the cases{\rm :}

$(1)\ n<0;$ \ $(2)\ i\not \in [0,d+s];$ \ $(3)\ i=0$ and $n\geq1;$ \ $(4)\ i=1$ and $n\geq2.$
\end{thm}
\end{rmk}

\begin{rmk}\label{rmkab}
For an abelian variety $A$, we know that for any $p,q$ and $0\leq i\leq 2g$,
\[
H_p(A,\Q(q))_{(i)}\simeq H^{2g-p}(A,\Q(g-q))^{(2g-i)}.
\]
Therefore Theorem \ref{vmh}\,(1) is the same as Theorem \ref{coh}, and Theorem \ref{vmh}(2) and (3) are the same.
\end{rmk}

The following corollary immediately follows from Corollary \ref{dec} and Theorem \ref{vmh}.
\begin{cor}[Theorem \ref{mth}]\label{decm}
Let notation be as in Theorem \ref{vmh}.
Then the decompositions in Corollary \ref{dec} are
\[
H_p(G,\Q(q))=\bigoplus_{i=i_0}^{i_1}H_p(G,\Q(q))_{(i)},
\quad
\ch^q(G,s;\Q)=\bigoplus_{i=j_0}^{j_1}\ch^q(G,s;\Q)^{(i)}.
\]
\end{cor}



\subsection{Proof of main result}
Let $G$ be a semiabelian variety over $k$, which is an extension of an abelian variety of dimension $g>0$ by a torus of rank $r$.

\subsubsection{Reduce to the case of an algebraically closed base field}
Let $L/k$ be a finite extension.
Let $G_L$ denote the scaler extension $G\tensor _k L$ and let $f_{L/k}: G_L \to G$ the projection.
Then the composition of the following homomorphism is the multiplication by the degree $[L:k]$:
\[
H_p(G,\Q(q))\stackrel{f_{L/k}^*}{\longrightarrow} H_p(G_L,\Q(q))\stackrel{{f_{L/k}}_*}{\longrightarrow} H_p(G,\Q(q)).
\]
Hence we obtain that the push- forward map ${f_{L/k}}^*$ is injective, and that there is a natural injection
\[
H_p(G,\Q(q)) \hookrightarrow \mathop{\rm colim}_{L/k} H_p(G_L,\Q(q)).
\]
By a result of a continuity of $\DM$ in the sense of Cisinski-D\'eglise \cite[\S 4.3, Proposition 4.3.4, Theorem 11.1.24]{CD}, we have a bijection
\[
\mathop{\rm colim}_{L/k} H_p(G_L,\Q(q)) \simeq H_p(G_{\bar{k}},\Q(q)).
\]
Here $\bar{k}$ denotes an algebraic closure of $k$.
Thus we have an injection
\[
H_p(G,\Q(q)) \hookrightarrow H_p(G_{\bar{k}},\Q(q)).
\]
Note that a similar assertion holds for $H_p(G,\Q(q))_{(i)}$, $\ch^q(G,s;\Q)$ and $\ch^q(G,s;\Q)^{(i)}$.
Therefore we may assume that $k$ is algebraic closed.

\subsubsection{In case of a torus}
Let $G=T$ be a torus of rank $r$.
By the above argument, we may assume that $T=\mb{G}_m^r$.
In this case, we know that
\[
\M(\mb{G}_m^r)\simeq \bigoplus_{i=0}^r \big(\Q(i)[i]\bigr)^{\oplus c_i},
\]
where $c_i=\binom{r}{i}$ denotes the binomial coefficient.
Thus, by isomorphism \ref{iso}, we obtain that for $0\leq i \leq r$,
\begin{align}
H_p(T,\Q(q))_{(i)}\simeq \ch^{i-q}(k,i+p-2q;\Q)^{\oplus c_i}.
\end{align}
From this equation and Theorem \ref{coh}, one can easily obtain the following proposition which induces Theorem \ref{vmh} for a torus.

\begin{prop}
Let $T$ be the torus $\mb{G}_m^r$ of rank $r$.
Let $p,q$ be integers.
Let $i$ be an integer with $0\leq i \leq r$.
Then $H_p(T,\Q(q))_{(i)}$ vanishes in the following cases{\rm :}
$(1)\ q>p;$ \ $(2)\ q>i;$
\ $(3)\ i<2q-p;$
\ $(4)\ q=i$ and $p-2q+i\geq 1;$
\ $(5)\ q=i-1$ and $p-2q+i \geq 2.$
\end{prop}


\subsubsection{In case of abelian varieties}
Let $G$ be an abelian variety $A$ of dimension $g$ over an algebraically closed field $k$.
By Remark \ref{rmkab}, it enough to prove Theorem \ref{vmh}\,(2).

\begin{lem}\label{WBF}
Let $\hat{A}$ be the dual abelian variety of $A$.
There is an isomorphism in $\DM$
\[
\M_1(\hat{A})\simeq (\M_1(A))^*(1)[2].
\]
Here $*$ denotes the dual in $\DM$.
\end{lem}
\begin{proof}
By Weil-Barsotti formula, we have an isomorphism
\[
\hat{A} \simeq \underline{\mathrm{Ext}}^1(A,\mb{G}_m).
\]
The assertion follows from a result of Barbieri-Viale--Kahn (\cite[\S 4]{BK}) on a comparison between Cartier dual of 1-motives and motivic dual of corresponding motives.
\end{proof}

\renewcommand{\proofname}{\bf Proof of Theorem \ref{vmh}\,(2)}
\begin{proof}
By Lemma \ref{WBF} and replacing $(A,\hat{A})$ by $(\hat{A},A)$, for an integer $i\geq0$, we have an isomorphism
\[
\sym^i(\M_1(A))\simeq \bigl(\sym^i(\M_1(\hat{A}))\bigr)^*(i)[2i].
\]
From this isomorphism, we have
\begin{align*}
H_p(A,\Q(q))_{(i)}
&=\HomD(\Q(q)[p],\sym^i(\M_1(A)))\\
&=\HomD(\Q(q)[p],(\sym^i(\M_1(\hat{A})))^*(i)[2i])\\
&=\HomD(\sym^i(\M_1(A)),\Q(i-q)[2i-p])\\
&=\ch^{i-q}(\hat{A},p-2q;\Q)^{(i)}\\
&\subset \ch^{i-q}(\hat{A},p-2q;\Q).
\end{align*}
From Theorem \ref{coh}, $\ch^{i-q}(\hat{A},p-2q;\Q)=0$ for $i-q \not \in [0, g+p-2q]$.
Thus $H_p(A,\Q(q))_{(i)}=0$ for $i \not \in [i_0,i_1]$, since $\sym^i(\M_1(A))=0$ for $i \not\in [0,2g]$ by Proposition \ref{vm}.
\end{proof}
\renewcommand{\proofname}{\bf Proof}



\subsubsection{In general case}

We prove Theorem \ref{vmh} by induction on rank of a semiabelian variety $G$.
Following \cite{EW}, we first give some triangles in $\DM$ which allows us to use the induction on the rank of a semiabelian variety.


We first give an exact sequence of  smooth group schemes.
Let $G$ be a semiabelian variety of rank $r$ over an algebraically closed field $k$.
Then there is an exact sequence of smooth group schemes over $k$ of the following form:
\[
1\to \mb{G}_m^r \stackrel{i}{\to} G \stackrel{f}{\to} A\to 1.
\]
Then, following \cite[Section 5.2]{EW}, we consider a cokernel $B$ of the composite map
\[
\mb{G}_m \to \mb{G}_m^r \stackrel{i}{\to} G,
\]
where the first map is the inclusion to the first factor of $\mb{G}_m^r$.
Then $B$ is a semiabelian variety of rank $r-1$ over $k$ which fits into
\begin{align}
1\to \mb{G}_m \to G \to B\to 1,\label{esg}\\
1\to \mb{G}_m^{r-1} \stackrel{i}{\to} B \stackrel{\bar{f}}{\to} A\to 1. \notag
\end{align}

\begin{lem}[\text{\cite{AEH,EW}}]\label{t2}
Let notation as above.
Then we have the following exact triangles in $\DM${\rm :}
\begin{align*}
&\M(B)(1)[1]\to\M(G)\to\M(B)\to\M(B)(1)[2].
\end{align*}
\end{lem}

\begin{proof}
From \eqref{esg}, we may regard $G$ as a $\mb{G}_m$-torsor on $B$.
Let $E$ be the line bundle over $B$ associated to the $\mb{G}_m$-torsor $G$.
Let $s: B \to E$ be the zero section.
By Gysin triangle attache to $E$ and $s(B)$, we have an exact triangle in $\DM$
\[
\M(E- s(B))\to \M(E)\to \M(s(B))(1)[2] \to \M(E- s(B))[1]
\]
Since $E-s(B)$ is isomorphic to $G$ and $\M(E)\simeq \M(B)\simeq \M(s(B))$, we get the desired triangle after shifting.
\end{proof}

\begin{lem}[\text{\cite{AEH,EW}}]\label{keyt}
Let notation as above.
Then we have the following exact triangles in $\DM${\rm :}
\begin{align*}
\sym^{n-1}(\M_1(B))(1)[1]\to\sym^n(\M_1(G))\to\sym^n(\M_1(B))\to\sym^{n-1}(\M_1(B))(1)[2].
\end{align*}
\end{lem}

\begin{proof}
By Corollary \ref{fm}, there is an isomorphism in $\DM$
\[
\sym^n(\M_1(G))\simeq [\sym^n(\widetilde{G})]_0. 
\]
Hence we may work in $\ShT(k,\Q)$ and we have an exact sequence in $\ShT(k,\Q)$:
\[
1\to \widetilde{\mb{G}}_m \to \widetilde{G} \to \widetilde{B} \to 1.
\]
Now we apply Lemma \ref{skey} to the above exact sequence in $\ShT(k,\Q)$.
Note that Proposition \ref{flat} allows us to apply Lemma \ref{skey} to the category $\ShT(k,\Q)$.
Then we have exact triangles for $i\geq 0$,
\begin{align*}
[\sym^n_{i+1}(\widetilde{G})]_0\to[\sym^n_i(\widetilde{G})]_0\to\sym^i(\M_1(\mb{G}_m))\tensor \sym^{n-i}(\M_1(B))\to[\sym^n_{i+1}(\widetilde{G})]_0[1].
\end{align*}
We know that $\M_1(\mb{G}_m)\simeq [\widetilde{\mb{G}}_m]_0\simeq \Q(1)[1]$ in $\DM$.
Therefore, by \cite[Proposition 2.3]{Bi}, $\sym^n(\M_1(\mb{G}_m))\simeq{\rm Alt}^n(\Q)(n)[n]=0$ for $n\geq2$ in $\DM$.
Thus we have
\[
[\sym^n_i(\widetilde{G})]_0
\simeq\begin{cases}
\sym^n(\M_1(G))&\text{if} \ i=0,\\
\sym^{n-1}(\M_1(B))(1)[1]&\text{if} \ i=1,\\
0& \text{otherwise.}
\end{cases}
\]
Hence we have the desired triangle.
\end{proof}



\renewcommand{\proofname}{\bf Proof of Theorem \ref{vmh}}
\begin{proof}
By Lemma \ref{t2}, we have an exact sequence
\[
\cdots\to H_{p-1}(B,\Q(q-1))\to H_p(G,\Q(q))\to H_p(B,\Q(q))\to\cdots.
\]
By Lemma \ref{keyt}, we also have exact sequences
\[
\cdots\to H_{p-1}(B,\Q(q-1))_{(i-1)}\to H_p(G,\Q(q))_{(i)}\to H_p(B,\Q(q))_{(i)}\to\cdots,
\]
\[
\cdots\leftarrow \ch^{q-1}(B,2q-p-1;\Q)^{(i-1)}\leftarrow \ch^q(G,2q-p;\Q)^{(i)}\leftarrow \ch^q(B,2q-p;\Q)^{(i)}\leftarrow\cdots.
\]
By these exact sequence, the assertion for a semiabelian variety $G$ follows from the induction hypothesis and the assertion for abelian varieties.
\end{proof}

\renewcommand{\proofname}{\bf Proof}

%
%

\section{0-th Suslin homology}
We here consider 0-th Susline homology $\h_0(G,\Z)$ of a semiabelian variety $G$, and proved Corollary \ref{mthm}(Corollary \ref{shom}).
We also give a description of  rational 0-th Suslin homology $\h_0(G,\Q)$ in terms of $K$-groups (Proposition \ref{repK}), using a result of Kahn--Yamazaki \cite{KY} (see Theorem \ref{tk}).

\subsection{Bloch's result for a semiabelian variety}

\begin{cor}[Corollary \ref{mthm}]\label{shom}
Let $G$ be a semiabelian variety over $k$, which is an extension of an abelian variety of dimension $g$ by a torus of rank $r$.
Let $I_G$ be the kernel of the degree map $\h_0(G,\Z)\to \Z$.
Then an iterated Pontryagin product $I_{G}^{*i}$ is torsion for $i>g+r$.

In particular, if $k$ is an algebraically closed field of characteristic $p\geq 0$, then for $i>g+r$,
\[
I_G^{*i}=0
\] 
up to $p$-torsion part.
Under the resolution of singularities, it holds without the assumption on the characteristic.
\end{cor}

\begin{proof}
From Theorem \ref{decm}, we have
\[
I_{G}\tensor\Q=\bigoplus_{n=1}^{g+r}H_0(G,\Q)_{(n)}.
\]
By compatibility of this decomposition with Pontryagin product (Corollary \ref{dec}),
\[
I_G^{*i}\tensor\Q \subset \bigoplus_{n=i}^{g+r}H_0(G,\Q)_{(n)}.
\]
Thus by Theorem \ref{vmh}\,(2), we obtain that $I_G^{*i}\tensor \Q=0$ for $i>g+r$.

For second assertion, we assume that $k$ is algebraically closed.
Then $I_G$ is generated by cycles of the form $[a]-[0_G]$, where $a,0_G\in G(k)$ and $0_G$ is the identity element of $G$.
Thus $I_G^{*2}$ is generated by cycles of the form
\[
[a+b]-[a]-[b]-[0_G]
\]
for $a,b\in G(k)$.
Let $\mathrm{alb}_G$ be the albanese map from $I_G \to G(k)$.
Then it easily see that
\[
\mathrm{alb}_G(I_G^{*2})=0_G.
\]
Since $I_G^{*i}\subset I_G^{*2}$ for $i>2$, we have $\mathrm{alb}_G(I_G^{*i})=0$ for $i>2$.
Now the second assertion follows from the first assertion and Theorem \ref{rt} below.
\end{proof}

\begin{thm}\label{rt}
Let $X$ be a reduced normal scheme, separated and of finite type over an algebraically closed field $k$ of characteristic $p\geq0$.
Then the albanese map
\[
\mathrm{alb}_X:\h_0(X,\Z)^0 \to Alb_X(k).
\]
from the degree-$0$-part of Suslin homology to the $k$-valued points of the albanese variety $Alb_X$ of $X$ induces an isomorphism on torsion groups up to $p$-torsion groups.
Under resolution of singularities, it holds without the assumption on the characteristic.
\end{thm}

\begin{rmk}
About Theorem \ref{rt}, Rojtman \cite{R} first proved in case where $X$ is projective smooth variety and for prime-to-$p$ torsions.
Bloch \cite[Theorem 4.2]{B3} (prime-to-$p$-part) and Milne \cite[Theorem 0.1]{Mi} ($p$-part) gave a cohomological proof of Rojtman's theorem.
For a quasi-projective smooth variety, the theorem is proved by Spie\ss-Szamuely \cite[Theorem 1.1]{SS} (prime-to-$p$-part) and Geisser \cite[Theorem 1.1]{Ge1} ($p$-part under resolution of singularities)
For a normal scheme, the theorem is proved by Barbieri-Viale--Kahn\cite[Corllary 14.5.3]{BK} ($char(k)=0$) and Geisser \cite[Theorem 1.1]{Ge2} ($char(k)\geq0$ and $p$-part under resolution of singularities).
\end{rmk}


\subsection{A description of rational Suslin homology}

We first briefly recall the definition of $K$-groups attached to semiabelian varieties.
Let $F_1,\dots, F_r$ be homotopy invariant Nisnevich sheaves with transfers.
We then define the $K$-group $K(k;F_1\dots,F_r)$ to be the quotient group
\begin{align}\label{defk}
\bigl(\bigoplus_{L/k: \text{finite}}F_1(L)\tensor\cdots\tensor F_r(L)\bigr)/R
\end{align}
where $R$ is a subgroup whose elements corresponds to Projection formula and Weil reciprocity.
For example, if $F_i=\widetilde{\mb{G}}_m$ for $i=1,\dots,r$, then we have the $r$-th Milnor $K$-group
\[
K(k;\widetilde{\mb{G}}_m,\dots,\widetilde{\mb{G}}_m)=K_r^M(k).
\] 
For the precise definition, see \cite{S,KY, IR}.
%

%
%

\begin{thm}[Kahn--Yamazaki \text{\cite{KY}}]\label{tk}
Let $F_1\dots,F_r$ be homotopy invariant Nisnevich sheaves with transfers .
Then there is an isomorphism
\[
K(k;F_1\dots,F_r)\simeq {\rm Hom}_{\boldsymbol{\rm DM}_{\rm -, Nis}^{\rm eff}(k,\Z)}\bigl(\Z,[F_1]_0\tensor\cdots\tensor [F_r]_0\bigr).
\]
\end{thm}

\begin{rmk}\label{rmkyk}
Ivorra-R\"ulling \cite{IR} defined ``$K$-groups" attached to NON-homotopy invariant sheaves with transfers which is equals to the above $K$-groups in case where all sheaves are homotopy invariant.
\end{rmk}

\begin{df}
(1)
For a semiabelian variety $G$ over a perfect field $k$, let $K_i(k;G)_{\Q}$ denotes the $K$-group $K(k; \widetilde{G},\dots,\widetilde{G})\tensor \Q$ attached to $i$ copies of the homotopy invariant Nisnevich sheaf with transfers $\widetilde{G}$ associated to $G$.

(2)
For $a_i\in G(L)$, let $\{a_1,\dots,a_i \}_{L/k}$ denote the element of $K_i(k;G)_{\Q}$ represented by $a_1\tensor \cdots \tensor a_i$.
We define an action of the permutation group $\Sigma_i$ on $K_i(k;G)$ as follows: for $\sigma \in \Sigma_i$ and $a_i \in G(L)$,
\[
\sigma(\{a_1,\dots,a_i \}_{L/k})=\{ a_{\sigma(1)},\dots,a_{\sigma(i)}\}_{L/k}. 
\]
Then we define $S_i(k;G)_{\Q}$ to be the image of an idempotent map $s^i:=\displaystyle \frac{1}{i!}\sum_{\sigma \in \Sigma_i}\sigma$:
\[
S_i(k;G)_{\Q}:={\rm Im}(s^i : K_i(k;G)_{\Q} \to K_i(k;G)_{\Q}).
\]
\end{df}

\begin{ex}
For $G=\mb{G}_m$ and an integer $i\geq0$, we have
\[
K_i(k;G)\simeq K_i^M(k).
\]
Here $K_i^M(k)$ is the $i$-th Milnor $K$-group of $k$.
\end{ex}

\begin{prop}[Gazaki \text{\cite{G}}(for an abelian variety)]\label{repK}
Let $G$ be a semiabelian variety over $k$, which is an extension of an abelian variety by a torus of rank $r$.
Let $\h_0(G,\Q)_{(i)}$ denote $H_0(G,\Q(0))_{(i)}$.
Then for any $0\leq i\leq r$, we have
\[
\h_0(G,\Q)_{(i)}\simeq S_i(k;G)_{\Q}.
\]
In particular,
\[
\h_0(G,\Q)\simeq \bigoplus_{i=0}^{g+r}S_i(k;G)_{\Q}
\]
\end{prop}


\begin{proof}
From Theorem \ref{tk}, we have the following commutative diagram 
\begin{align*}
\xymatrix{
K_i(k;G)_{\Q}\ar[d]^{s^i}\ar@{=}[r]&\HomD(\Q,\M_1(G)^{\tensor i})\ar[d]^{s^i_{\M_1(G)}}\ar@{=}[r]&\bigl((\widetilde{G})^{\tensor_{\rm HI}^i}\bigr)(k)\ar[d]^{s^i_{\widetilde{G}}}\\
K_i(k;G)_{\Q}\ar@{=}[r]&\HomD(\Q,\M_1(G)^{\tensor i})\ar@{=}[r]&\bigl((\widetilde{G})^{\tensor_{\rm HI}^i}\bigr)(k).
}
\end{align*}
 Since $\h_0(G,\Q)^{(i)}=\HomD(\Q,\sym^i(\M_1(G)))$, the assertion follows from this diagram.
\end{proof}

\begin{rmk}\label{shk}
Let notation be as in Propostion \ref{repK}. 
Let $F^n\h_0(G)$ denote $\bigoplus _{i=n}^{g+r}S_i(k;G)_{\Q}$.
Then the filtration $F^{\bullet}\h_0(G)$ on $\h_0(G,\Q)$ satisfies the following{\rm :}
\begin{enumerate}[{\rm (a)}]
\item $F^1\h_0(G)=\ker({\rm deg}: \h_0(G)\to \Q)=I_{G,\Q}${\rm ;}

\item $F^2\h_0(G)=\ker \bigl({\rm alb}_{G/k}:F^1\h_0(G) \to G(k)\tensor \Q\bigr)$.
This map is induced by the albanese map $G \to Alb_G${\rm ;}

\item $F^n\h_0(G)*F^m\h_0(G)\subset F^{n+m}\h_0(G)$.
Here $*$ denote the Pontryagin product{\rm ;}

\item $(F^1\h_0(G))^{*n}=0$ for $n>g+r$ (see Corollary \ref{shom}).
\end{enumerate}
Bloch \cite{B} studied a filtration on the Chow group $\ch_0(A)$ of an abelian variety over an algebraically closed field, which is defined by iterated Pontryagin product $I_A^{*r}$ of the kernel $I_A$ of the degree map $\ch_0(A) \to \Z$. 
A similar filtration on $\ch_0(A)$ for an abelian variety over a field is studied by Gazaki \cite{G}, using $K$-groups attached to $A$.
\end{rmk}

\bigskip

\noindent\textbf{Acknowledgements}

The author expresses his gratitude to Professors Thomas Geisser and Marc Levine for helpful suggestions and comments.
He also thanks Giuseppe Ancona for a lot of valuable discussions and comments for the preliminary version of this paper and for Remark \ref{repft}.
He thanks Professora Annette Huber and Bruno Kahn for a valuable discussion about Appendix B.
Thanks are also due to Federico Binda, Jin Cao, Utsav Choudhury and Bradely Drew for helpful comments.
%

\appendix
\def\thesection{\Alph{section}}


\section{An exact sequence for symmetric products}

A purpose of this section is to prove Lemma \ref{skey} below which gives an auxiliary exact sequence related to symmetric products.
This section follows from \cite{EW} and \cite{AEH}.

\begin{df}
Let $\ml{A}$ be an abelian $\Q$-linear tensor category.
Assume that the tensor product $\tensor$ is exact and that arbitrary direct sum exists in $\ml{A}$.
Let $0 \to F\stackrel{f}{\to} G\to H \to 0$ be an exact sequence in $\ml{A}$.
We define a subobject $\sym^n_i(G)$ of $\sym^n(G)$ as follows: for an integer $i$ with $0\leq i \leq n$,
\[
\sym_i^n(G):={\rm Im}\bigl(s^n_G\circ (f^{\tensor i}\tensor {\rm id}_G^{\tensor (n-i)}): F^{\tensor i}\tensor G^{\tensor (n-i)} \to G^{\tensor n}\bigr).
\]
Here $s_G^n$ denotes the canonical projection $G^{\tensor n} \to \sym^n(G)$.

For $i>n$, put $\sym_i^n(G)=0$.

By definition, there is an inclusion map $\sym_{i+1}^n(G) \to \sym^n_i(G)$.
\end{df}

\begin{lem}\label{skey}
Let $\ml{A}$ be an abelian $\Q$-linear tensor category.
Assume that the tensor product $\tensor$ is exact and that arbitrary direct sum exists in $\ml{A}$.
Let $0 \to F\stackrel{f}{\to} G\to H \to 0$ be an exact sequence in $\ml{A}$.
Then for any $n\geq 1$ and $i$ with $0\leq i\leq n$, we have the following exact sequences
\[
0\to \sym_{i+1}^n(G) \to \sym^n_i(G)\to \sym^{i}(F)\tensor \sym^{n-i}(H) \to 0.
\]
\end{lem}

\newcommand{\F}{\ml{F}}
We prepare some notation to prove this lemma.
For a map $F\to G$ in $\ml{A}$, let $\F^i_F(G)$ denote
\[
\F_F^i(G)=
\begin{cases}
G & \text{if}\ \  i=0\\
F & \text{if}\ \  i=1\\
0 & \text{otherwise}
\end{cases}
\]

For $n>1$, we define a filtration $\F_F^i(G^{\tensor n})$ on $G^{\tensor n}$ as follows:
\[
\F_F^i(G^{\tensor n}):={\rm Im}\bigl(\bigoplus _{j_1+\cdots +j_n=i} \F_F^{j_1}(G)\tensor \cdots \F_F^{j_n}(G)\to G^{\tensor n}\bigr).
\]
By definition, there is an inclusion $\F_F^{i+1}(G^{\tensor n})\to \F_F^i(G^{\tensor n})$.
We write $gr_F^i(G^{\tensor n})$ for the quotient 
\[
gr_F^i(G^{\tensor n}):=\F_F^i(G^{\tensor n})/\F_F^{i+1}(G^{\tensor n}).
\]
Since the inclusion is $\Sigma_i$-equivariant, we have an exact sequence
\[
0\to \sym_{i+1}^n(G) \to \sym^n_i(G)\to s^n_G(gr^i_F(G^{\tensor n})) \to 0.
\]

\renewcommand{\proofname}{\it Proof of Lemma \ref{skey}}
\begin{proof}
Let $0 \to F\to G\to H \to 0$ be an exact sequence in $\ml{A}$.
Now it suffices to show that 
\[
gr^i_F(G^{\tensor n})\simeq \bigoplus_{j_1+\cdots +j_n=i}\F_F^{j_1}(H)\tensor \cdots \tensor \F_F^{j_n}(H) \subset (F\oplus H)^{\tensor n},
\]
since $\sym^i(F)\tensor\sym^{n-i}(H)$ is the image of $\displaystyle \bigoplus_{j_1+\cdots +j_n=i}\F_F^{j_1}(H)\tensor \cdots \tensor \F_F^{j_n}(H)$ under the projection $s_{F\oplus H}^n:(F\oplus H)^{\tensor n}\to \sym^n(F\oplus H)$.
Note that $\displaystyle \bigoplus_{j_1+\cdots +j_n=i}\F_F^{j_1}(H)\tensor \cdots \tensor \F_F^{j_n}(H)$ is a direct sum of all tensor products of $i$ copies of $F$ and $n-i$ copies of $H$. 

This is proved by induction on $n$.
It is clear in case $n=1$ by the definition of the filtration.
It is also clear in case $i=0$.

Consider the following commutative diagram with exact rows:
\begin{align*}
\xymatrix{
0 \to \F_F^{i+1}(G^{\tensor n})\tensor F \ar[r]\ar@{^{(}->}[d] & \F_F^i(G^{\tensor n}) \tensor F\ar[r] \ar@{^{(}->}[d] &gr_F^i(G^{\tensor n})\tensor F\ar@{^{(}->}[d] \to 0\\
0 \to \F_F^{i+1}(G^{\tensor n}) \tensor G\ar[r] & \F_F^i(G^{\tensor n})\tensor G \ar[r]  & gr_F^i(G^{\tensor n})\tensor G \to 0
}
\end{align*}
Here we used the exactness of the tensor $\tensor$ to have injections in the diagram.
By the above diagram, we have
\[
\F_F^{i+1}(G^{\tensor n})\tensor F=\ker \bigl( \bigl(\F_F^i(G^{\tensor n})\tensor F\bigr)\oplus \bigl(\F_F^{i+1}(G^{\tensor n})\tensor G \bigr) \to \F_F^{i+1}(G^{\tensor n})\tensor G \bigr).
\]
From this equality and the commutative diagram
\begin{align*}
\xymatrix{
\bigl(\F_F^i(G^{\tensor n})\tensor F\bigr)\oplus \bigl(\F_F^{i+1}(G^{\tensor n})\tensor G \bigr)\ar@{->>}[r]\ar[d] &\F_F^{i+1}(G^{\tensor n+1})\ar@{^{(}->}[d]\\
\F_F(G^{\tensor n})\tensor G \ar@{^{(}->}[r]&G^{\tensor n+1},
}
\end{align*}
we obtain the following commutative diagram with exact rows: for $i\geq 1$, 
\begin{align*}
\xymatrix{
0 \to \F_F^{i+1}(G^{\tensor n})\tensor F \ar[r]\ar@{^{(}->}[d] & \bigl(\F_F^i(G^{\tensor n})\tensor F\bigr)\oplus \bigl(\F_F^{i+1}(G^{\tensor n})\tensor G \bigr)\ar[r] \ar@{^{(}->}[d] &\F_F^{i+1}(G^{\tensor n+1})\ar@{^{(}->}[d] \to 0\\
0 \to \F_F^i(G^{\tensor n})\tensor F \ar[r] &\bigl(\F_F^{i-1}(G^{\tensor n})\tensor F\bigr)\oplus \bigl(\F_F^{i}(G^{\tensor n})\tensor G\bigr) \ar[r]  & \F_F^i(G^{\tensor n+1}) \to 0.
}
\end{align*}
The injectivity of some maps in this diagram follows from the exactness of the tensor.
By snake lemma and the induction hypothesis, we have an exact sequence
\begin{align*}
0&\to \bigoplus_{j_1+\cdots +j_n=i}\F_F^{j_1}(H)\tensor \cdots \tensor \F_F^{j_n}(H) \tensor F\\
&\to \bigl(\bigoplus_{j_1+\cdots +j_n=i-1}\F_F^{j_1}(H)\tensor \cdots \tensor \F_F^{j_n}(H) \tensor F\bigr)
\oplus \bigl(\bigoplus_{j_1+\cdots +j_n=i}\F_F^{j_1}(H)\tensor \cdots \tensor \F_F^{j_n}(H)\tensor G\bigr)\\
& \to gr_F^i(G^{\tensor n+1}) \to 0.
\end{align*}
Therefore we have the desired isomorphism
\begin{align*}
gr_F^i(G^{\tensor n+1})&\simeq \bigl(\bigoplus_{j_1+\cdots +j_n=i-1}\F_F^{j_1}(H)\tensor \cdots \tensor \F_F^{j_n}(H) \tensor F\bigr)\\
&\qquad \oplus \bigl(\bigoplus_{j_1+\cdots +j_n=i}\F_F^{j_1}(H)\tensor \cdots \tensor \F_F^{j_n}(H)\tensor H\bigr)\\
&=\bigoplus_{j_1+\cdots +j_{n+1}=i}\F_F^{j_1}(H)\tensor \cdots \tensor \F_F^{j_{n+1}}(H).
\end{align*}
\end{proof}


\section{Remarks on tensor product on $\ShT(k,\Q)$}

We here make two remarks about the tensor product $\tensor_{\rm ShT}$ on $\ShT(k,\Q)$
One is that the tensor product is exact (Proposition \ref{flat}).
A key of a proof of the exactness are results of Suslin--Voevodsky \cite{SV0,SV} and Cisinski--D\'eglise \cite{CD}.
Another is to give a example of \'etale sheaves with transfers for which \'etale sheafification of presheaf tensor product is not isomorphic to the tensor product $\tensor_{\rm ShT}$. 

Let us fix notation:
\begin{itemize}
\item 
${\bf \ml{S}ch}/k$ : the category of separated schemes of finite type over $k$;

\item
${\bf \ml{S}m}/k$ : the subcategory of ${\bf \ml{S}ch}/k$ of smooth $k$-schemes;

\item
$\PST(k,\Q)$ : the category of presheaves on ${\bf \ml{S}m}/k$ of $\Q$-modules with transfers;

\item
$\ShT(k,\Q)$ : the category of Nisnevich sheaves on ${\bf \ml{S}m}/k$ of $\Q$-modules with transfers.

\item
$\underline{\bf Sh}_{\rm Nis}(k,\Q)$ : the category of Nisnevich sheaves on ${\bf\ml{S}ch}/k$ of $\Q$-modules;

\item
$\nis$ : the category of Nisnevich sheaves on ${\bf\ml{S}ch}/k$ of $\Q$-modules with transfers;

\item
$\qfh$ : the category of qfh-sheaves on ${\bf\ml{S}ch}/k$ of $\Q$-modules. 
\end{itemize}

\subsection{Exactness of the tensor product on $\ShT(k,\Q)$}

We recall the tensor product on $\PST$:
For $X\in {\bf\ml{S}m}/k$, $L(X)$ denotes the representable presheaf with transfers. 
We first define the tensor product $L(X)\tensor_{{\rm PST}}L(Y)$ as
\[
L(X)\tensor_{{\rm PST}}L(Y):=L(X\times Y).
\]
Let $F$ be presheaf with transfers. 
We have a canonical projective resolution $\ml{L}(F)\to F$ of $F$ of the following form
\begin{align*}
\cdots &\to \bigoplus_j L(Y_j) \to \bigoplus_i L(X_i) \to F\to 0
\end{align*}
Then for presheaves with transfers $F$ and $G$, the tensor product $F\tensor_{{\rm PST}}G$ of $F$ and $G$ is defined to be
\[
H_0\bigl({\rm Tot}(\ml{L}(F)\tensor_{\rm PST} \ml{L}(G))\bigr).
\]
For Nisnevich sheaves with transfers $F$ and $G$, the tensor product $F\tensor_{{\rm ShT}}G$ is given by 
\[
F\tensor_{{\rm ShT}}G=\bigl(F\tensor_{{\rm PST}}G\bigr)_{\rm Nis}.
\]


\begin{prop}\label{flat}
The bifunctor $\tensor_{\rm ShT}$ on $\ShT(k,\Q)$ is exact.
\end{prop}

\begin{rmk}\label{rmkt}
(1) The tensor product on $\ShT(k,\Z)$ is right exact, since internal Hom objects exist in $\ShT(k,\Z)$.

(2) If F is a presheaf of $\Q$-modules with transfers,  then $F_{\rm Nis}=F_{\text{\'et}}$.
Therefore, Proposition \ref{flat} holds for $\boldsymbol{\rm ShT}_{\text{\'et}}(k,\Q)$.
\end{rmk}

From Proposition \ref{flat} and the definition of tensor product on the derived category $D^{-}(\ShT(k,\Q))$ of $\ShT(k,\Q)$, we obtain the following:

\begin{cor}\label{fm}
Let $F$ and $G$ be Nisnevich sheaves of $\Q$-modules with transfers.
Let $[F]_0$ denote the complex consisting of $F$ concentrated in degree zero.
Then 

{\rm (1)} $[F]_0\tensor [G]_0$ is isomorphic to $[F\tensor_{{\rm ShT}}G]_0$ in $D^{-}(\ShT(k,\Q))$.

{\rm (2)} $\sym^i([F]_0)$ is isomorphic to $[\sym^i(F)]_0$ in $D^{-}(\ShT(k,\Q))$.
\end{cor}

We prepare some lemmas and recall results of Suslin--Voevodsky \cite[Corollary 6.6, Theorem 6.7]{SV0}\cite[Theorem 4.2.12]{SV} (cf. Cisinski--D\'eglise \cite[Theorem 10.5.5]{CD}) to prove Proposition \ref{flat}.

\begin{thm}
Let $X\in{\bf\ml{S}ch}/k$ be a separated scheme of finite type over $k$.
Let $\underline{L}(X)_{\Q}\in \nis$ denote the Nisnevich sheaf with transfers represented by $X$.
Then $\underline{L}(X)_{\Q}$ is a qfh-sheaf.

Furthermore, let $\Q_{\rm qfh}(X)\in \qfh$ denote the qfh-sheaf represented by $X$.
Then $\underline{L}(X)_{\Q}$ is isomorphic to the qfh-sheaf $\Q_{\rm qfh}(X)$.
\end{thm}

By this theorem,  for any qfh-sheaf $F \in \qfh$, Cisinski--D\'eglise \cite[\S 10]{CD} defined a Nisnevich sheaf with transfers $\rho(F)$ as follows: for any $X\in {\bf\ml{S}ch}/k$,
\[
\rho(F)(X):= {\rm Hom}_{{\rm Sh}_{\rm qfh}(k,\Q)}(\underline{L}(X)_{\Q}, F).
\]
Then we have a functor
\[
\rho : \qfh \longrightarrow \nis.
\]
Now we have a functor
\[
\psi :\qfh \stackrel{\rho}{\to} \nis \stackrel{\varphi^*}{\to} \ShT(k,\Q).
\]
where the second functor is the pull-back with respect to $\varphi : {\bf\ml{S}m}/k \to {\bf\ml{S}ch}/k$.

We also have a functor 
\[
\phi : \ShT(k,\Q) \to \nis \to \underline{\bf Sh}_{\rm Nis}(k,\Q) \longrightarrow \qfh
\]
where the first functor is the left adjoint of $\varphi^*$, the second one is the forgetful functor and the third one is the qfh-sheafification.
Cisinski-D\'eglise showed the following properties of functors $\psi, \phi$:

\begin{prop}[\text{\cite[Proposition 10.5.14]{CD}}]\label{cdp}
The following conditions are true:
\begin{enumerate}[{\rm (1)}]
\item For any smooth $k$- scheme $X$, $\psi (\Q_{\rm qfh}(X))\simeq L(X)_{\Q}${\rm ;}

\item The functor $\psi$ admits a left adjoint $\phi${\rm ;}

\item For any smooth $k$-scheme $X$, $\phi (L(X)_{\Q})\simeq \Q_{\rm qfh}(X)${\rm ;}

\item The functor $\psi$ is exact and preserves colimits{\rm ;}

\item The functor $\phi$ is exact, fully faithful and preserves colimits
\end{enumerate}

\end{prop}

\medskip
\renewcommand{\proofname}{\bf Proof of Proposition \ref{flat}}
\begin{proof}
Since the tensor product on $\ShT(k,\Q)$ is right exact (cf. Remark \ref{rmkt}), it suffices to show that the tensor product is left exact.
Thus, we need to show that for any injection $F_1 \to F_2$ in $\ShT(k,\Q)$ and any objects $G,H$ of $\ShT(k,\Q)$, the map
\[
{\rm Hom}_{\ShT}(H,F_1\tensor_{\rm ShT} G)\to {\rm Hom}_{\ShT}(H,F_2\tensor_{\rm ShT} G)
\]
is injective.
Since the functor $\phi$ is fully faithful by Proposition \ref{cdp}\,(5), we have the following commutative diagram
\[
\xymatrix{
{\rm Hom}_{\ShT(k,\Q)}(H,F_1\tensor_{\rm ShT} G)\ar[r]\ar@{=}[d]& \ar@{=}[d]{\rm Hom}_{\ShT(k,\Q)}(H,F_2\tensor_{\rm ShT} G)\\
{\rm Hom}_{\qfh}(\phi(H),\phi(F_1\tensor_{\rm ShT} G))\ar[r]& {\rm Hom}_{\qfh}(\phi(H),\phi(F_2\tensor_{\rm ShT} G))
}
\]
Here if $\phi$ is a tensor functors, then $\phi(F_i\tensor_{\rm ShT} G)=\phi(F_i)\tensor \phi(G)$ for $i=1,2$.
Hence the bottom horizontal map is
\[
{\rm Hom}_{\qfh}(\phi(H),\phi(F_1)\tensor \phi(G))\to {\rm Hom}_{\qfh}(\phi(H),\phi(F_2)\tensor \phi(G)).
\]
But, this map is injective because the functor $\phi$ is exact by Proposition \ref{cdp}\,(5) and the tensor on $\qfh$ is exact.
Now our task is to show that $\phi$ is a tensor functor.

Let $F,G\in \ShT(k,\Q)$ be Nisnevich sheaves with transfers.
In case where $F=L(X)_{\Q}$ and $G=L(Y)_{\Q}$, we have
\begin{align*}
\phi(L(X)_{\Q}\tensor_{\rm ShT}L(Y)_{\Q})
&=\phi(L(X\times Y)_{\Q})\\
&\simeq\Q_{\rm qfh}(X\times Y)\\
&=\Q_{\rm qfh}(X)\tensor\Q_{\rm qfh}(Y)\simeq\phi(L(X)_{\Q})\tensor\phi(L(Y)_{\Q}).
\end{align*}
In general case, let $\ml{L}(F)$ and $\ml{L}(G)$ be the canonical projective resolution of $F$ and $G$ respectively.
Then from the the above special case and the exactness of $\phi$, we obtain that
\begin{align*}
\phi(F\tensor_{\rm ShT}G)
&=\phi(H_0({\rm Tot}(\ml{L}(F)\tensor_{\rm ShT}\ml{L}(G))))\\
&=H_0({\rm Tot}(\phi(\ml{L}(F))\tensor\phi(\ml{L}(G))))\\
&=\phi(F)\tensor\phi(G).
\end{align*}
Thus the functor $\phi$ is a tensor functor.
\end{proof}
\renewcommand{\proofname}{\bf Proof}

\subsection{Example}

We prove that in general, the \'etale sheafification of a presheaf tensor $L(X)\tensor L(Y)$ can not be isomorphic to $L(X\times_k Y)$.

We have a map
\[
\phi:L(X)\tensor L(Y)\to L(X\times_k Y).
\]
and consider the \'etale stalks at a strictly Henselian local scheme $S$ , which is
\[
\phi:\Cor (S,X)\tensor \Cor(S,Y) \to \Cor(S,X\times_k Y), \quad Z\tensor W \mapsto [Z\times_SW].
\]

\begin{prop}\label{uniso}
Let $X=Y=\mathbb{A}^1_k$.
Let $S$ be a strictly localization of $\mathbb{A}^1_k$ at the origin.
Let $t$ be a uniformizer of $S$.
Then $\phi$ is not surjective.
\end{prop}

Write $X=\mathrm{Spec}(k[x])$ and $Y=\mathrm{Spec}(k[y])$.
We define $f(x)\in S[x]$ as
\[
f(x)=\begin{cases}
x^2-t& \text{if $\mathrm{char}(k)\not =2$},\\
x^3-t & \text{if $\mathrm{char}(k) =2$.}
\end{cases}
\] 
and define cycles to be
\[
Z:=\mathrm{Spec} (S[x]/(f(x))) \in \Cor(S,X) \quad \text{and} \quad W:=\mathrm{Spec} (S[y]/(f(y))) \in \Cor(S,Y).
\]
Then we have
\[
T:=Z\times _SW=\mathrm{Spec}(S[x,y]/(f(x),f(y))).
\]
And we put $T_1$ and $T_2$ as follows: in case $\mathrm{char}(k)\not =2$, 
\[
T_1:=\mathrm{Spec}\bigl(S[x,y]/(f(x),x-y)\bigr), \ T_2:=\mathrm{Spec}\bigl(S[x,y]/(f(x),x+y)\bigr);
\]
in case $\mathrm{char}(k)=2$,
\[
T_1:=\mathrm{Spec}\bigl(S[x,y]/(f(x),x-y)\bigr), \ T_2:=\mathrm{Spec}\bigl(S[x,y]/(f(x),x^2+xy+y^2)\bigr).
\] 
Then, one easily see that
\[
\phi(Z\tensor W)=[T]=T_1+T_2\in \Cor(S, X\times _kY) .
\]

\begin{lem}\label{exlem}
Let notation be as in above.
A pair $(\alpha,\beta)$ of irreducible cycles $\alpha\in \Cor(S,X)$ and $\beta \in \Cor(S,Y)$ such that $\phi(\alpha\tensor\beta)$ contains $T_1($or $T_2)$ is only $(Z,W)$.
\end{lem}
\begin{proof}
Irreducible cycles $\alpha$ and $\beta$ are given by monic irreducible polynomials $g(x)\in S[x]$ and $h(y) \in S[y]$ respectively, i.e.
\[
\alpha=\mathrm{Spec}(S[x]/(g(x)))\subset S\times X,\quad \beta=\mathrm{Spec}(S[y]/(h(y)))\subset S\times Y.
\]
Put 
\[
\gamma:=\alpha\times _S\beta=\mathrm{Spec}(S[x,y]/(g(x),h(y))) \subset S\times X\times Y.
\]
Suppose that our cycle $T_1\subset \gamma$ is a irreducible component.

Let $p,q$ denote the projections from $S\times X\times Y$ to $S\times X$ and $S\times Y$ respectively.
Then, 
\[
p(T_1)=Z,\quad  q(T_1)=W.
\]
Hence the images of $\gamma$ along $p$ and $q$ contain $Z$ and $W$ respectively, i.e.
\[
Z \subset \mathrm{Spec}(S[x]/(g(x))) \quad \text{and}\quad W\subset \mathrm{Spec}(S[y]/(h(y))).
\]
This inclusions implies that
\[
g(x)=a(x)f(x) \quad \text{and} \quad h(y)=b(y)f(y).
\]
Since $f,g$ and $h$ are monic and irreducible, we have
\[
g(x)=f(x),\quad h(y)=f(y).
\]
Hence $\alpha=Z,\beta=W$.

The above argument works also for $T_2$. 
\end{proof}

\renewcommand{\proofname}{\bf Proof of Proposition \ref{uniso}}
\begin{proof}
Assume that $\phi$ is surjective.
Then we have an element $x:=\sum n_{ij}(\alpha_i\tensor \beta_j) \in \Cor(S,X)\tensor\Cor(S,Y)$ such that $\phi(x)=T_1$.
For some component $\alpha_i\tensor \beta_j$ of $x$, $\phi(\alpha_i\tensor\beta_j)$ contains $T_1$.
By Lemma \ref{exlem}, such $\alpha_i\tensor \beta_j$ is $Z\tensor W$ only.
Therefore we have
\[
x=Z\tensor W+\sum n_{ij}(\alpha_i\tensor \beta_j),
\]
where the sum is taken over all $i,j$ such that $(\alpha_i,\beta_j)\not =(Z,W)$.
Then
\[
T_1=\phi(x)=T_1+T_2+\sum n_{ij}\phi(\alpha_i\tensor \beta_j).
\]
Hence 
\[
T_2=-\sum n_{ij}\phi(\alpha_i\tensor \beta_j).
\]
But this equality cannot be happen by Lemma \ref{exlem}.
\end{proof}


\bigskip


Fakult\"at Mathematik, Universit\"at Duisburg-Essen, Thea-Leymann-Stra\ss e 9, 45127 Essen, Germany\medskip

\textit{E-mail address}: rin.sugiyama@uni-due.de

\end{document}